\documentclass[10pt]{amsart}
\usepackage{amssymb}%,MnSymbol,color,times}
\usepackage{amsthm,amsmath}
\usepackage{cite}

\title{Associative string functions}
\author{Erkko Lehtonen}
\address{Centro de \'Algebra da Universidade de Lisboa, Avenida Professor Gama Pinto 2, 1649-003 Lisboa, Portugal}
\address{Departamento de Matem\'atica, Faculdade de Ci\^encias, Universidade de Lisboa, 1749-016 Lisboa, Portugal} \email{erkko[at]campus.ul.pt}

\author{Jean-Luc Marichal}
\address{Mathematics Research Unit, FSTC, University of Luxembourg \\
6, rue Coudenhove-Kalergi, L-1359 Luxembourg, Luxembourg} \email{jean-luc.marichal[at]uni.lu}

\author{Bruno Teheux}
\address{Mathematics Research Unit, FSTC, University of Luxembourg \\
6, rue Coudenhove-Kalergi, L-1359 Luxembourg, Luxembourg} \email{bruno.teheux[at]uni.lu}

\date{Ocober 16, 2014}

\begin{document}

\theoremstyle{plain}
\newtheorem{theorem}{Theorem}[section]% Supprimer [section] pour une numérotation linéaire
\newtheorem{lemma}[theorem]{Lemma}
\newtheorem{proposition}[theorem]{Proposition}
\newtheorem{corollary}[theorem]{Corollary}
\newtheorem{fact}[theorem]{Fact}
\newtheorem*{main}{Main Theorem}

\theoremstyle{definition}
\newtheorem{definition}[theorem]{Definition}
\newtheorem{example}[theorem]{Example}

\theoremstyle{remark}
\newtheorem*{conjecture}{Conjecture}
\newtheorem{remark}{Remark}
\newtheorem{claim}{Claim}

\newcommand{\N}{\mathbb{N}}
\newcommand{\Q}{\mathbb{Q}}
\newcommand{\R}{\mathbb{R}}

\newcommand{\ran}{\mathrm{ran}}
\newcommand{\dom}{\mathrm{dom}}
\newcommand{\id}{\mathrm{id}}
\newcommand{\med}{\mathrm{med}}
\newcommand{\ofo}{\mathrm{ofo}}
\newcommand{\Ast}{\boldsymbol{\ast}}

\newcommand{\bfu}{\mathbf{u}}
\newcommand{\bfv}{\mathbf{v}}
\newcommand{\bfw}{\mathbf{w}}
\newcommand{\bfx}{\mathbf{x}}
\newcommand{\bfy}{\mathbf{y}}
\newcommand{\bfz}{\mathbf{z}}

\newcommand{\length}[1]{{\vert #1 \vert}}

\newcommand\restr[2]{{% we make the whole thing an ordinary symbol
  \left.\kern-\nulldelimiterspace % automatically resize the bar with \right
  #1 % the function
% \vphantom{\big|} % pretend it's a little taller at normal size
  \right|_{#2} % this is the delimiter
  }}

\begin{abstract}
We introduce the concept of associativity for string functions, where a string function is a unary operation on the set of strings over a given alphabet. We discuss this new property and describe certain classes of associative string functions. We also characterize the recently introduced preassociative functions as compositions of associative string functions with injective unary maps. Finally, we provide descriptions of the classes of associative and preassociative functions which depend only on the length of the input.
\end{abstract}

\keywords{Associativity, preassociativity, string function, functional equation, axiomatization}

\subjclass[2010]{20M05, 20M32, 39B72, 68R99}

\maketitle

%---------------------------------------------------------------------------------------------- Section 1
\section{Introduction}

Throughout this paper, $X$ denotes a nonempty set, called the \emph{alphabet,} and its elements are called \emph{letters}. The symbol $X^*$ stands for the free monoid $\bigcup_{n \geqslant 0} X^n$ generated by $X$, and its elements are called \emph{strings}, where the empty string $\varepsilon$ is such that $X^0=\{\varepsilon\}$. Thus, we assume that $X^*$ is endowed with the concatenation operation for which the empty string $\varepsilon$ is the neutral element. We denote the elements of $X^*$ by bold roman letters $\bfx$, $\bfy$, $\bfz$. If we want to stress that such an element is a letter of $X$, we use non-bold italic letters $x$, $y$, $z$, $\ldots$ For every string $\bfx$ and every integer $n\geqslant 0$, the power $\bfx^n$ stands for the string obtained by concatenating $n$ copies of $\bfx$. In particular, we have $\bfx^0=\varepsilon$. The notation $\bfx^*$ stands for the set of all powers of $\bfx$. The \emph{length} of a string $\bfx$ is denoted by $|\bfx|$. In particular, we have $|\varepsilon|=0$.

Let $Y$ be a nonempty set. Recall that, for every integer $n\geqslant 0$, a function $F\colon X^n\to Y$ is said to be \emph{$n$-ary}. We say that a function $F\colon X^*\to Y$ has an \emph{indefinite arity} or, more simply, is \emph{variadic} or \emph{$\Ast$-ary} (pronounced ``star-ary''). In particular, we say that a variadic function $F \colon X^* \to X^*$ is a \textit{string function} over the alphabet $X$. Finally, we say that a variadic function $F\colon X^*\to Y$ is \emph{standard} if it satisfies the condition
$$%\begin{equation}\label{eq:PA-ori}
F(\bfx) = F(\varepsilon)\quad\iff\quad\bfx = \varepsilon.
$$%\end{equation}

We now introduce the associativity property for string functions. Equivalent conditions are given in Proposition~\ref{prop:8sdf7zz}.

\begin{definition}\label{de:AssStr}
We say that a string function $F \colon X^* \to X^*$ is \emph{associative} if it satisfies the condition
\begin{equation}\label{eq;assoc}
F(\bfx \bfy \bfz) = F(\bfx F(\bfy) \bfz){\,},\qquad \bfx,\bfy,\bfz \in X^*.
\end{equation}
\end{definition}

Data processing can be seen as the computation of string functions. Many commonplace data processing tasks correspond to associative string functions, e.g., sorting data in alphabetical order, transforming a string of letters into upper case. In this context, associativity may be a desirable property because it allows one to work locally on small pieces of data at a time. For instance, the string function which corresponds to sorting the letters of every string in alphabetical order is associative and standard whereas the string function which consists in removing from every string all occurrences of a given letter is associative but not standard.

\begin{fact}\label{fact:45u}
Let $F\colon X^*\to X^*$ be an associative function. Then the following conditions hold.
\begin{enumerate}
\item[(a)] $F$ is idempotent w.r.t.\ composition, i.e., we have $F\circ F=F$ (take $\bfx\bfz=\varepsilon$ in Eq.~(\ref{eq;assoc})).

\item[(b)] If $F$ is standard, then $F(\varepsilon)=\varepsilon$ (here associativity can be replaced by the weaker condition $F(F(\varepsilon))=F(\varepsilon)$).

\item[(c)] If $F$ is not standard, then there exists $\mathbf{a}\in X^*\setminus\{\varepsilon\}$ such that $F(\bfx\bfz)=F(\bfx\mathbf{a}\bfz)$ for every $\bfx,\bfz\in X^*$.
\end{enumerate}
\end{fact}

%Special instances of associative string functions include the standard variadic functions $F\colon X^* \to X\cup\{\varepsilon\}$, which have been recently investigated in \cite{MarTeh}. According to Fact~\ref{fact:45u}, these functions satisfy $F(\varepsilon)=\varepsilon$.

In Section 2 of this paper we investigate the associativity property (\ref{eq;assoc}) for string functions. In particular, we provide different equivalent definitions of this property. We also investigate the subclass of associative functions $F\colon X^*\to X^*$ satisfying the condition $|F(\bfx)|\leqslant m$ for every $\bfx\in X^*$, where $m$ is a fixed nonnegative integer (when $m=1$ this subclass consists of the associative variadic functions $F\colon X^* \to X\cup\{\varepsilon\}$, whose standard versions have been investigated in \cite{MarTeh}). In Section 3 we investigate the class of preassociative functions, which was recently introduced in \cite{MarTeh}. In particular we characterize these functions as compositions of associative string functions with injective unary maps. Finally, in Section 4 we provide descriptions of the classes of associative and preassociative functions which depend only on the length of the input.

The following notation will be used in this paper. We let $\N$ denote the set of nonnegative integers. For every $n\in\N$ and for every function $F\colon X^*\to Y$, we denote by $F_n$ the \emph{$n$-ary part of $F$,} i.e., the restriction $\restr{F}{X^n}$ of $F$ to the set $X^n$. The domain and range of any function $f$ are denoted by $\dom(f)$ and $\ran(f)$, respectively. The identity function on any nonempty set is denoted by $\id$.

%---------------------------------------------------------------------------------------------- Section 2
\section{Associative functions}

As mentioned in the introduction, in this section we investigate associativity for string functions. Clearly, the identity function on $X^*$ is associative and standard. The following two examples provide nontrivial instances of associative functions.

\begin{example}[Letter removing]\label{ex:LetterRemov56}
Let $a\in X$ be fixed. Let the map $F_a \colon X^* \to X^*$ be defined inductively by $F_a(z) = z$ if $z \neq a$, $F_a(a) = \varepsilon$, and $F_a(\bfx z) = F_a(\bfx)F_a(z)$. Let also the map $G_a\colon X^*\to X^*$ be defined by $G_a(\bfx)=a$, if $\bfx\in a^*$, and $G_a(\bfx)=F_a(\bfx)$, if $\bfx\notin a^*$. Then both $F_a$ and $G_a$ are associative but not standard. Moreover, $F_a(\varepsilon)=\varepsilon$ and $G_a(\varepsilon)=a\neq\varepsilon$.
\end{example}

\begin{example}[Duplicate removing]\label{ex:DuplRemov56}
Define the function $\ofo \colon X^* \to X^*$ by the following procedure.
Given a string $\bfx\in X^*$, delete all repeated occurrences of letters, keeping only the first occurrence of each letter; the resulting string is $\ofo(\bfx)$.
In other words, the function $\ofo$ outputs the letters of its input in the \emph{order of first occurrence} (hence the acronym ofo).
For example,
\begin{align*}
\ofo(indivisibilities) &= indvsblte{\,}, \\
\ofo(subdermatoglyphic) &= subdermatoglyphic{\,}.
\end{align*}
It is easy to verify that the function $\ofo$ is associative and standard.
\end{example}

%As already observed, it is natural for an associative function $F\colon X^*\to X^*$ to satisfy $F(\varepsilon)=\varepsilon$. Indeed, if $F(\varepsilon)=\mathbf{a}$ for some $\mathbf{a}\in X^*\setminus\{\varepsilon\}$, then we obtain $F(\bfx\bfz)=F(\bfx\mathbf{a}^n\bfz)$ for every $\bfx,\bfz\in X^*$ and every $n\in\N$, which shows that $F(\varepsilon)$ in a sense behaves like the empty string $\varepsilon$.

The following result, which was already established in \cite{CouMar11} for associative and standard functions $F\colon X^* \to X\cup\{\varepsilon\}$, gives equivalent definitions of associativity under the condition $F(\varepsilon)=\varepsilon$. Note that this latter condition need not hold for associative functions (to give an example, take any constant function whose value is distinct from $\varepsilon$).

\begin{proposition}\label{prop:8sdf7zz}
Let $F \colon X^* \to X^*$ be a function such that $F(\varepsilon)=\varepsilon$. The following conditions are equivalent.
\begin{enumerate}
\item[(i)] $F$ is associative.

\item[(ii)] For any $\bfx, \bfy, \bfz, \bfx', \bfy', \bfz' \in X^*$ such that $\bfx \bfy \bfz = \bfx' \bfy' \bfz'$ we have $F(\bfx F(\bfy) \bfz) = F(\bfx' F(\bfy') \bfz')$.

\item[(iii)] For any $\bfx, \bfy, \bfz \in X^*$ we have $F(F(\bfx \bfy ) \bfz) = F(\bfx F(\bfy\bfz ))$.

\item[(iv)] For any $\bfx, \bfy \in X^*$ we have $F(\bfx \bfy) = F(F(\bfx) F(\bfy))$.
\end{enumerate}
\end{proposition}

\begin{proof}
(i) $\implies$ (ii) $\implies$ (iii). Trivial.

(iii) $\implies$ (iv). Taking $\bfy\bfz=\varepsilon$ shows that $F$ satisfies $F\circ F=F$. Taking $\bfx=\varepsilon$ and then $\bfz=\varepsilon$, we obtain $F(\bfx F(\bfy))=F(F(\bfx)\bfy)=F(F(\bfx\bfy))=F(\bfx\bfy)$ and therefore $F(F(\bfx) F(\bfy))=F(\bfx\bfy)$.

(iv) $\implies$ (i). $F$ clearly satisfies $F\circ F=F$ (take $\bfy=\varepsilon$). Repeated applications of (iv) then give
\begin{multline*}
F(\bfx F(\bfy) \bfz)
= F(F(\bfx F(\bfy)) F(\bfz))
= F(F(F(\bfx) F(F(\bfy))) F(\bfz)) \\
= F(F(F(\bfx) F(\bfy)) F(\bfz))
= F(F(\bfx \bfy) F(\bfz))
= F(\bfx \bfy \bfz),
\end{multline*}
which completes the proof.
\end{proof}

The following proposition shows that the definition of associativity remains unchanged if the length of the
string $\bfx\bfz$ is bounded above by one.

\begin{proposition}\label{prop:xzleq1}
A function $F \colon X^* \to X^*$ is associative if and only if $F(\bfx \bfy \bfz) = F(\bfx F(\bfy) \bfz)$ for any $\bfx,\bfy,\bfz \in X^*$ such that $\length{\bfx \bfz} \leqslant 1$.
\end{proposition}

\begin{proof}
Necessity is obvious. For sufficiency,
assume that $F(\bfx \bfy \bfz) = F(\bfx F(\bfy) \bfz)$ for any $\bfx,\bfy,\bfz \in X^*$ such that $\length{\bfx \bfz} \leqslant 1$.
We prove by induction on $\length{\bfx \bfz}$ that $F(\bfx \bfy \bfz) = F(\bfx F(\bfy) \bfz)$ holds for all $\bfx,\bfy,\bfz \in X^*$.
The basis of the induction is clear from our assumption.
Assume that the claim holds if $\length{\bfx \bfz} = k$ for some $k \geqslant 1$.
Let $\bfx,\bfy,\bfz \in X^*$ be such that $\length{\bfx \bfz} = k + 1$.
If $\length{\bfx} \geqslant 1$, then $\bfx = a \bfx'$ for some $a \in X$, $\bfx' \in X^*$, with $\length{\bfx'} = \length{\bfx} - 1$, and we have
\begin{align*}
F(\bfx F(\bfy) \bfz) =
F(a F(\bfx' F(\bfy) \bfz)) =
F(a F(\bfx' \bfy \bfz)) =
F(\bfx \bfy \bfz),
\end{align*}
where the first and the third equalities hold by our assumption, and the second equality holds by the induction hypothesis since $\length{\bfx' \bfz} = \length{\bfx \bfz} - 1 = k$.
A similar argument shows that $F(\bfx F(\bfy) \bfz) = F(\bfx \bfy \bfz)$ if $\length{\bfz} \geqslant 1$.
This completes the proof, because at least one of $\bfx$ and $\bfz$ is nonempty.
\end{proof}

It is noteworthy that any associative function $F\colon X^*\to X^*$ satisfies the following equation
\begin{equation}\label{eqn:bla2}
F(x_1 \cdots x_n) ~=~ F(F(x_1 \cdots x_{n-1})x_n){\,}, \qquad n \geqslant 1,
\end{equation}
or equivalently,
\begin{equation}\label{eqn:bla}
F(x_1 \cdots x_n) ~=~ F(F( \cdots F(F(\varepsilon) x_1) \cdots ) x_n){\,}, \qquad n \geqslant 1.
\end{equation}

Equation~\eqref{eqn:bla2} clearly shows that every associative function $F\colon X^*\to X\cup\{\varepsilon\}$ is completely determined by its nullary, unary, and binary parts (note however that if $F_0(\varepsilon)\neq\varepsilon$, then we have $F_1(x)=F_2(F_0(\varepsilon)x)$ and hence $F_1$ is also determined by $F_0$ and $F_2$). The following proposition gives an extension of this observation to string functions.

\begin{definition}
Let $D$ be a nonempty set and let $m\in\N$. We say that a map $F \colon D \to X^*$ is \emph{$m$-bounded} if $\length{F(x)}\leqslant m$ for every $x\in D$.
\end{definition}

For instance, the $1$-bounded string functions are exactly the functions $F\colon X^*\to X\cup\{\varepsilon\}$.

\begin{proposition}\label{prop:mB}
Let $F \colon X^* \to X^*$ be an associative function and let $m\in\N$.
\begin{enumerate}
\item[(a)] $F$ is $m$-bounded if and only if $F_0,\ldots,F_{m+1}$ are $m$-bounded.
\item[(b)] If $F$ is $m$-bounded, then $F$ is uniquely determined by its parts of arity at most $m + 1$, i.e., if $G \colon X^* \to X^*$ is an associative $m$-bounded function such that $G_i = F_i$ for $i = 0, \ldots, m + 1$, then $F = G$.
\end{enumerate}
\end{proposition}

\begin{proof}
(a)
Necessity is trivial.
For sufficiency, assume that $F_0,\ldots,F_{m+1}$ are $m$-bounded.
We show by induction on $k$ that $F_k$ is $m$-bounded. Assume that $F_k$ is $m$-bounded for some $k \geqslant m + 1$.
Let $\bfx \in X^{k+1}$.
By associativity we have $F_{k+1}(x_1 \cdots x_{k+1}) = F(F_k(x_1 \cdots x_k) x_{k+1})$.
Since $F_k$ is $m$-bounded we have $\length{F_k(x_1 \cdots x_k) x_{k+1}}\leqslant m+1$, and so $F(F_k(x_1 \cdots x_k) x_{k+1})=F_j(F_k(x_1 \cdots x_k) x_{k+1})$ for some $j\in\{1,\ldots,m+1\}$. Since $F_j$ is $m$-bounded, we have $\length{F(F_k(x_1 \cdots x_k) x_{k+1})} \leqslant m$ and hence $F_{k+1}$ is $m$-bounded.

(b)
Let $F\colon X^*\to X^*$ and $G\colon X^*\to X^*$ be associative $m$-bounded functions such that $G_i = F_i$ for $i = 0, \ldots, m + 1$.
We show by induction on $k$ that $F_k=G_k$ for all $k\in\N$. Assume that $F_k=G_k$ for some $k \geqslant m + 1$.
Let $\bfx \in X^{k+1}$.
We then have
\begin{multline*}
G_{k+1}(x_1 \cdots x_{k+1})
= G(G_k(x_1 \cdots x_k) x_{k+1})
= G(F_k(x_1 \cdots x_k) x_{k+1}) \\
= F(F_k(x_1 \cdots x_k) x_{k+1})
= F_{k+1}(x_1 \cdots x_{k+1}),
\end{multline*}
where
the first equality holds by associativity of $G$,
the second equality holds by the inductive hypothesis,
the third equality holds since $F$ is $m$-bounded and by the inductive hypothesis,
and the last equality holds by associativity of $F$.
We conclude that $F = G$.
\end{proof}

%\textcolor{red}{
%We are still looking for an example of a string-associative function $F\colon X^*\to X^*$ for which there is a $m > 0$ such that $F$ is \emph{not} $m$-bounded and such that for any string-associative function $G\colon X^* \to X^*$
%\[
%G_0=F_0 {\,},{\,} \ldots {\,},{\,} G_{m+1}=F_{m+1} \quad\implies\quad F=G.
%\]
%}

Setting $m=1$ in Proposition~\ref{prop:mB}(a) and using Fact~\ref{fact:45u}(b), we immediately obtain the following corollary.

\begin{corollary}
An associative and standard function $F\colon X^*\to X^*$ ranges in $X\cup\{\varepsilon\}$ if and only if $\ran(F_1)\subseteq X$ and $\ran(F_2)\subseteq X$.
\end{corollary}

The following result gives necessary and sufficient conditions for an $m$-bounded function $F \colon X^* \to X^*$ to be associative. This result was established in \cite[Proposition~3.3]{MarTeh} in the special case of standard functions $F \colon X^* \to X\cup\{\varepsilon\}$ satisfying the condition $F(\varepsilon)=\varepsilon$.

\begin{proposition}\label{prop:NSCfor-mBSA}
Let $m\in\N$. An $m$-bounded function $F \colon X^* \to X^*$ is associative if and only if the following conditions are satisfied.
\begin{enumerate}
\item[(a)] $F\circ F_k=F_k$ for $k=0,\ldots,m+1$.
\item[(b)] $F(x)=F(xF(\varepsilon))$ for all $x\in X$.
\item[(c)] $F(F(x\bfy)z)=F(xF(\bfy z))$ for all $x\in X$, $\bfy\in X^*$, and $z\in X$ such that $\length{x\bfy z}\leqslant m+2$.
\item[(d)] Condition \eqref{eqn:bla2} or condition \eqref{eqn:bla}  holds.
\end{enumerate}
\end{proposition}

\begin{proof}
Conditions (a)--(d) clearly follow from associativity.
We prove that conditions (a)--(d) are sufficient.
By conditions (b)--(d) and Proposition~\ref{prop:xzleq1}, it is enough to show that $F\circ F=F$ and that $F(x\bfy z) = F(xF(\bfy z))$ for all $x\bfy z \in X^*$ such that $|x\bfy z|> m+2$. For the second assertion, we proceed by induction on $k=|x\bfy z|$. Assume that the condition holds for some $k\geqslant m+2$ and let $u\in X$. We then have
$$
F(x\bfy zu)
~=~ F(F(x\bfy z)u)
~=~ F(F(xF(\bfy z))u)
~=~ F(xF(F(\bfy z)u))
~=~ F(xF(\bfy zu)){\,},
$$
where the first equality is obtained by condition (d) and the other equalities by the induction hypothesis, condition (c), and the fact that $F$ is $m$-bounded.

It remains to prove that $F\circ F=F$, or equivalently, $F\circ F_k=F_k$ for every $k\in\N$. According to condition (a), we may assume that $k \geqslant m + 2$. Setting $\bfx = \bfy z$ such that $\length{\bfx}\geqslant m+2$, we have
\[
F(\bfx) = F(\bfy z) = F(F(\bfy) z) = F(F(F(\bfy) z)) = F(F(\bfy z))= F(F(\bfx)),
\]
where the second and the fourth equality are obtained by condition (d) and the third by condition (a) and the fact that $F$ is $m$-bounded.
\end{proof}

The following important result immediately follows from Proposition~\ref{prop:NSCfor-mBSA}. It gives necessary and sufficient conditions on $F_0,\ldots,F_{m+1}$ for an $m$-bounded function $F\colon X^*\to X^*$ to be associative.

\begin{theorem}\label{thm:NSC-1st}
Let $m\in\N$. For $k=0,\ldots,m+1$, let $F_k\colon X^k\to X^*$ be an $m$-bounded function. Then there exists an associative $m$-bounded function $G\colon X^*\to X^*$ such that $G_k=F_k$ for $k=0,\ldots,m+1$ if and only if conditions (a)--(c) of Proposition~\ref{prop:NSCfor-mBSA} hold. Such a function is then uniquely determined by the condition $G(\bfy z)=G(G(\bfy)z)$ for every $\bfy\in X^*$ and every $z\in X$.
\end{theorem}

\begin{remark}
Let $F\colon X^*\to X^*$ be an associative $m$-bounded function and let $k\in\{0,\ldots,m\}$. From Proposition~\ref{prop:NSCfor-mBSA} it follows that if we replace $F_j$ with the identity function on $X^j$ for $j=0,\ldots,k$, then the resulting function is still associative and $m$-bounded.
\end{remark}

It is clear that the identity function on $X^*$ is an associative and standard function that is not $m$-bounded for any $m\in\N$. The following examples provide other instances of associative and standard functions that are not $m$-bounded.

\begin{example}\label{ex:unbounded}
Let $|$ be a fixed letter of the alphabet $X$, and define the string function $F\colon X^* \to X^*$ by the following procedure:
given an input string, insert the letter $|$ between any two consecutive letters neither one of which is $|$.
For example,
$$
F(a) ~=~ a{\,},\quad F(ab) ~=~ F(a|b) ~=~ a|b{\,},\quad F(||) ~=~ ||{\,},\quad F(||ab|||cd) ~=~ ||a|b|||c|d{\,}.
$$
It is an easy exercise to verify that the function $F$ is associative and standard.
It is also clear that $F$ is not $m$-bounded for any $m\in\N$.
\end{example}

\begin{example}\label{ex:lengthpres}
Let $m\in\N$ and let $c \in X$. Assume that $F \colon X^* \to X^*$ is an associative function that satisfies $|F_k(\bfx)| = k$ for $k=0,\ldots,m$. The function $G \colon X^* \to X^*$ defined by $G_0 = F_0, \dots, G_m = F_m$, and $G_k = c^k$ for every $k\geqslant m+1$, is associative and standard.
\end{example}

\begin{remark}
It is an open problem whether Example~\ref{ex:lengthpres} would remain true if we replaced the equality $|F_k(\bfx)| = k$ with the inequality $|F_k(\bfx)| \leqslant k$.
\end{remark}

We end this section by investigating the associative functions which are injective. Actually, as the following result shows, associative functions are never injective (except the identity function) and hence cannot be used as coding functions.

\begin{proposition}\label{prop:klm}
If $F \colon X^* \to X^*$ is injective and satisfies $F = F \circ F$, then it is equal to the identity. In particular, any associative and injective function $F \colon X^* \to X^*$ is equal to the identity.
\end{proposition}

\begin{proof}
Applying $F^{-1}$ to both sides of $F=F\circ F$ immediately shows that $F$ is the identity function. The second statement follows immediately by Fact~\ref{fact:45u}(a).
\end{proof}

\begin{remark}
Proposition~\ref{prop:klm} can be refined as follows. Suppose that $F \colon X^* \to X^*$ satisfies $F = F \circ F$, $\ran(F_k)\subseteq X^k$, and $F_k$ is injective for some $k\in\N$. Then $F_k=\restr{\id}{X^k}$.
\end{remark}

Proposition~\ref{prop:klm} raises the question of measuring how far an associative function different from the identity is from being injective. The following proposition shows that such a function is in a sense highly non-injective.

\begin{definition}
Let $\preceq$ be the quasiorder (i.e., reflexive and transitive binary relation) defined on the set of string functions by setting $F\preceq G$ if $\ker(G)\subseteq \ker(F)$, that is,
$$
G(\bfx) = G(\bfy) \quad\implies\quad F(\bfx) = F(\bfy),\qquad \bfx,{\,}\bfy\in X^*.
$$
%If $F \preceq G$ we say that $F$ is \emph{coarser} than $G$ or that $G$ is \emph{finer} than $F$.
We denote by $\prec$ the irreflexive part of $\preceq$.
\end{definition}

\begin{proposition}
Let $F\colon X^* \to X^*$ be an associative function different from the identity. Then there is an infinite sequence of associative functions $(F^m\colon X^*\to X^*)_{m\geqslant 1}$ such that $F\preceq F^1\prec F^2\prec \cdots \prec \id$.
\end{proposition}

\begin{proof}
First, we note that there exists $(\bfx_0,\bfx_1)\in \ker(F)$ such that $\bfx_0\neq \bfx_1$ and $\varepsilon\not\in\{\bfx_0, \bfx_1\}$. Indeed, since $F$ is not injective there exists $(\bfy_0, \bfy_1)\in \ker(F)$ such that $\bfy_0\neq\bfy_1$. If $\bfy_0=\varepsilon$, it follows that for every $\bfx\in X^*$, we have $F(\bfx)=F(\bfx \varepsilon)=F(\bfx F(\varepsilon))=F(\bfx F(\bfy_1))=F(\bfx \bfy_1)$. Therefore, we can choose $(\bfx_0, \bfx_1)=(\bfx, \bfx\bfy_1)$, where $\bfx\neq\varepsilon$.

For any integer $m\geqslant 0$, denote by $\theta_m$ the equivalence relation defined as follows: we say that two strings are $\theta_m$-equivalent if one can be obtained from the other by substituting some occurrences of $\bfx_0^{2^m}$ with $\bfx_1^{2^m}$ and some occurrences of $\bfx_1^{2^m}$ with $\bfx_0^{2^m}$. It follows that $\theta_{m+1}\subset \theta_m$ for every $m\geqslant 0$ and by definition $\theta_0\subseteq\ker(F)$.

For every integer $m\geqslant 1$ we denote by $\pi_m$ the quotient map $\pi_m\colon X^*\to X^*/\theta_m$ and we let $g_m\colon X^*/\theta_m \to X^*$ be a map satisfying $g_m(\bfx/\theta_m) \in \bfx/\theta_m$. Let us prove that the sequence $(F^m\colon X^*\to X^*)_{m\geqslant 1}$ defined as $F^m=g_m\circ\pi_m$ satisfies the conditions of the statement. It is clear that $F^m$ maps any $\bfx\in X^*$ to a distinguished element in $\bfx/\theta_m$. Hence by definition we have $\ker(F^m)=\theta_m$ for all $m\geqslant 1$. Thus, it remains the prove that the functions $F^m$ are associative.

Let $m\geqslant 1$ and let $\bfx,\bfy,\bfz\in X^*$. Since $F^m(\bfy)=(g_m\circ\pi_m)(\bfy)$ we obtain that $\bfy$ and $F^m(\bfy)$ are $\theta_m$-equivalent. It follows easily from the definition of $\theta_m$ that the strings $\bfx\bfy\bfz$ and $\bfx F^m(\bfy)\bfz$ are $\theta_m$-equivalent, that is, $F^m(\bfx\bfy\bfz)=F^m(\bfx F^m(\bfy)\bfz)$.
\end{proof}

%---------------------------------------------------------------------------------------------- Section 3
\section{Preassociative functions}

The concept of preassociativity has been recently introduced in \cite{MarTeh} for standard variadic functions $F\colon X^*\to Y$. Actually, this concept can be immediately applied to every variadic function $F\colon X^*\to Y$.

\begin{definition}
We say that a function $F\colon X^*\to Y$ is \emph{preassociative} if it satisfies the condition
$$%\begin{equation}\label{eq:PA}
F(\bfy) = F(\bfy') \quad\implies\quad F(\bfx \bfy \bfz) = F(\bfx \bfy' \bfz){\,},\qquad \bfx,\bfy,\bfy',\bfz\in X^*.
$$%\end{equation}
\end{definition}

\begin{fact}\label{fact:45u2}
Let $F\colon X^*\to Y$ be a preassociative function. If $F$ is not standard, then there exists $\mathbf{a}\in X^*\setminus\{\varepsilon\}$ such that $F(\bfx\bfz)=F(\bfx\mathbf{a}\bfz)$ for every $\bfx,\bfz\in X^*$.
\end{fact}

\begin{example}\label{ex:SPA64}
The function $F\colon X^*\to\N$ defined by $F(\bfx)=|\bfx|$ (number of letters in $\bfx$) is preassociative and standard. For every $a\in X$, the function $F\colon X^*\to\N$ defined by $F(\bfx)=|F_a(\bfx)|$ (number of letters in $\bfx$ distinct from $a$), where $F_a$ is defined in Example~\ref{ex:LetterRemov56}, is preassociative but not standard. For every $a\in X$, the function $F\colon X^*\to\N$ defined by $F(\bfx)=|G_a(\bfx)|$, where $G_a$ is defined in Example~\ref{ex:LetterRemov56}, is not preassociative. Indeed, for every $b\in X\setminus\{a\}$, we have $F(b)=F(\varepsilon)=1$ but $F(b^2)=2\neq 1=F(b)$. The function $F\colon X^*\to\N$ defined by $F(\bfx)=|\ofo(\bfx)|$ (number of distinct letters in $\bfx$), where $\ofo$ is defined in Example~\ref{ex:DuplRemov56}, is not preassociative. Indeed, for distinct $a,b\in X$, we have $F(a)=F(b)=1$ but $F(aa)=1\neq 2=F(ab)$. Finally, for every $a\in X$, the functions $F_a$ and $G_a$ are preassociative but not standard. The function $\ofo$ is preassociative and standard.
\end{example}

\begin{remark}
Example~\ref{ex:SPA64} motivates the following open question. Find necessary and sufficient conditions on an associative function $F\colon X^*\to X^*$ for the function $\bfx\mapsto |F(\bfx)|$ to be preassociative.
\end{remark}

The following two assertions are straightforward adaptations of results reported in \cite[Propositions~4.3 and 4.5]{MarTeh}.

\begin{proposition}\label{prop:MT4.5}
Let $F\colon X^*\to Y$ be a preassociative (resp.\ preassociative and standard) function and let $g\colon Y\to Y'$ be a function. If $\restr{g}{\ran(F)}$ is injective, then the function $H\colon X^*\to Y'$ defined as $H=g\circ F$ is preassociative (resp.\ preassociative and standard).
\end{proposition}

\begin{proposition}\label{prop:preidem}
A function $F \colon X^* \to X^*$ is associative if and only if it is preassociative and satisfies $F = F \circ F$.
\end{proposition}

We now define a new concept which will prove to be closely related to $m$-bounded string functions (see Proposition~\ref{prop:Hmb-Fmr51}).

\begin{definition}
Let $m\in\N$. We say that a map $F\colon X^*\to Y$ has an \emph{$m$-determined range} if $\ran(F)=\bigcup_{k=0}^m\ran(F_k)$.
\end{definition}

We immediately observe that the property of having an $m$-determined range is preserved under left composition with unary maps: if $F\colon X^*\to Y$ has an $m$-determined range, then so has $g\circ F$ for any map $g\colon Y\to Y'$, where $Y'$ is an nonempty set.

\begin{proposition}\label{prop:Hmb-Fmr51}
Let $m\in\N$. Any map $F\colon X^*\to Y$ satisfying $F=F\circ H$, where $H\colon X^*\to X^*$ is $m$-bounded, has an $m$-determined range.
\end{proposition}

\begin{proof}
Let $F\colon X^*\to Y$ be a function satisfying $F=F\circ H$, where $H\colon X^*\to X^*$ is $m$-bounded, and let $\bfx\in X^*$. Since $H$ is $m$-bounded, there exists $k\in\{0,\ldots,m\}$ such that $F(\bfx)=(F\circ H)(\bfx)=(F_k\circ H)(\bfx)$. Therefore, we have $\ran(F)\subseteq\bigcup_{k=0}^m\ran(F_k)$. Since the other inclusion is obvious, $F$ has an $m$-determined range.
\end{proof}

We now give a characterization of the preassociative (resp.\ preassociative and standard) functions $F\colon X^*\to Y$ as compositions of the form $F=f\circ H$, where $H\colon X^*\to X^*$ is associative (resp.\ associative and standard) and $f\colon\ran(H)\to Y$ is injective. This result answers a question raised in \cite{MarTeh} and is stated in Theorem~\ref{factorizationPA} below.

First recall that a function $g$ is a \emph{quasi-inverse} \cite[Sect.~2.1]{SchSkl83} of a function $f$ if
$$
f\circ \restr{g}{\ran(f)}=\restr{\id}{\ran(f)}\qquad\mbox{and}\qquad\ran(\restr{g}{\ran(f)})=\ran(g).
$$
The set of quasi-inverses of a function $f$ is denoted by $Q(f)$. Under the assumption of the Axiom of Choice (AC), the set $Q(f)$ is nonempty for any function $f$. In fact, the Axiom of Choice is just another form of the statement ``every function has a quasi-inverse''. Note also that the relation of being quasi-inverse is symmetric: if $g \in Q(f)$ then $f \in Q(g)$; moreover, we have $\ran(g)\subseteq\dom(f)$ and $\ran(f)\subseteq\dom(g)$ and the functions $\restr{f}{\ran(g)}$ and $\restr{g}{\ran(f)}$ are injective.

The following two lemmas are extensions of results reported in \cite[Proposition~2.2 and Lemma~4.8]{MarTeh}.

\begin{lemma}\label{lem:gfr}
Assume AC and let $F \colon X^* \to Y$ be a function. For any $g \in Q(F)$, define the function $H \colon X^* \to X^*$ by $H= g \circ F$. Then the following conditions hold.
\begin{enumerate}
\item[(a)] We have $F = F \circ H$, $H = H \circ H$, and the map $\restr{F}{\ran(H)}$ is injective.

\item[(b)] $F$ is standard if and only if so is $H$. In either case, we have $H(\varepsilon)=\varepsilon$.
\end{enumerate}
Moreover, if $F$ has an $m$-determined range for some $m\in\N$, then $g$ can always be chosen so that $\ran(g)\subseteq\bigcup_{k=0}^mX^k$ and therefore $H$ is $m$-bounded. Conversely, if $H$ is $m$-bounded for some $m\in\N$, then $F$ has an $m$-determined range.
\end{lemma}

\begin{proof}
By definition of $H$ we have $F \circ H$ $=$ $F \circ g \circ F$ $=$ $F$, $H \circ H$ $=$ $g \circ F \circ g \circ F$ $=$ $g \circ F$ $=$ $H$, and the map $\restr{F}{\ran(g)}=\restr{F}{\ran(H)}$ is injective. Now, if $F$ is standard, then from the identity $F(H(\varepsilon))=F(\varepsilon)$ we immediately derive $H(\varepsilon)=\varepsilon$. Moreover, if $H(\bfx)=\varepsilon$, then we have $F(\bfx)=F(H(\bfx))=F(\varepsilon)$ and therefore $\bfx=\varepsilon$, which shows that $H$ is standard. Conversely, if $H$ is standard, then from the identity $F(\bfx)=F(\varepsilon)$ we obtain $H(\bfx)=(g\circ F)(\bfx)=(g\circ F)(\varepsilon)=H(\varepsilon)$ and therefore $\bfx=\varepsilon$, which shows that $F$ is standard.

Now, if $F$ has an $m$-determined range for some $m\in\N$, then there always exists $g\in Q(F)$ such that $\ran(g)\subseteq\bigcup_{k=0}^mX^k$; indeed, if $y\in\ran(F_k)$ for some $k\in\{0,\ldots,m\}$, then we can take $g(y)\in F_k^{-1}\{y\}\subseteq X^k$. Therefore $H=g\circ F$ is $m$-bounded. Conversely, if $H$ is $m$-bounded for some $m\in\N$, then $F$ has an $m$-determined range by Proposition~\ref{prop:Hmb-Fmr51}.
\end{proof}

\begin{lemma} \label{lem:bxd}
Assume AC and let $F \colon X^* \to Y$ be a function. The following assertions are equivalent.
\begin{enumerate}
\item[(i)] $F$ is preassociative (resp.\ preassociative and standard).
\item[(ii)] For every $g \in Q(F)$, the function $H \colon X^* \to X^*$ defined by $H = g \circ F$ is associative (resp.\ associative and standard).
\item[(iii)] There is $g \in Q(F)$ such that the function $H \colon X^* \to X^*$ defined by $H = g \circ F$ is associative (resp.\ associative and standard).
\end{enumerate}
For any $m\in\N$, the same equivalence holds if we add the condition that $F$ has an $m$-determined range in assertion (i) and the conditions $\ran(g)\subseteq\bigcup_{k=0}^mX^k$ and $H$ is $m$-bounded in assertions (ii) and (iii).
\end{lemma}

\begin{proof}
(i) $\implies$ (ii).
Let $g \in Q(F)$ and $H = g \circ F$.
We know by Lemma~\ref{lem:gfr} that $H = H \circ H$.
Since $\restr{g}{\ran(F)}$ is injective, we have that $H$ is preassociative (resp.\ preassociative and standard) by Proposition~\ref{prop:MT4.5}.
It follows from Proposition~\ref{prop:preidem} that $H$ is associative (resp.\ associative and standard).

(ii) $\implies$ (iii).
Trivial.

(iii) $\implies$ (i).
By Proposition~\ref{prop:preidem}, $H$ is preassociative (resp.\ preassociative and standard).
Since $\restr{g}{\ran(F)}$ is an injective function from $\ran(F)$ onto $\ran(g) = \ran(H)$, we have $F = (\restr{g}{\ran(F)})^{-1} \circ H$. It follows from Proposition~\ref{prop:MT4.5} that $F$ is preassociative (resp.\ preassociative and standard).

The last part of the result follows from Lemma~\ref{lem:gfr}.
\end{proof}

\begin{theorem}\label{factorizationPA}
Assume AC and let $F \colon X^* \to Y$ be a function. The following conditions are equivalent.
\begin{enumerate}
\item[(i)] $F$ is preassociative (resp.\ preassociative and standard).
\item[(ii)] There exists an associative (resp.\ associative and standard) function $H \colon X^* \to X^*$ and an injective function $f \colon \ran(H) \to Y$ such that $F = f \circ H$.
\end{enumerate}
Moreover, we have the following.
\begin{enumerate}
\item[(a)] If condition (ii) holds, then we have $f = \restr{F}{\ran(H)}$, $f^{-1}\in Q(F)$, and we may choose $H = g \circ F$ for any $g \in Q(F)$.

\item[(b)] For any $m\in\N$, the equivalence between (i) and (ii) still holds if we add the condition that $F$ has an $m$-determined range in assertion (i) and the condition that $H$ is $m$-bounded in assertion (ii). In this case the condition $\ran(g)\subseteq\bigcup_{k=0}^mX^k$ must be added in statement (a).
\end{enumerate}
\end{theorem}

\begin{proof}
(i) $\implies$ (ii).
Let $H\colon X^*\to X^*$ be defined by $H=g\circ F$, where $g\in Q(F)$. By Lemma~\ref{lem:gfr} we have $F=f\circ H$, where $f=\restr{F}{\ran(H)}$ is injective. By Lemma~\ref{lem:bxd}, $H$ is associative (resp.\ associative and standard).

(ii) $\implies$ (i).
By Proposition~\ref{prop:preidem} we have that $H$ is preassociative (resp.\ preassociative and standard). Then also $F$ is preassociative (resp.\ preassociative and standard) by Proposition~\ref{prop:MT4.5}.

(a) If condition (ii) holds, then $F\circ H=f\circ H\circ H=f\circ H$ and hence $\restr{F}{\ran(H)}=\restr{f}{\ran(H)}=f$. Moreover, since $f$ is injective we have $H=f^{-1}\circ F$ and hence $F\circ f^{-1}\circ F=F\circ H=f\circ H\circ H=f\circ H=F$, which shows that $f^{-1}\in Q(F)$.

(b) Follows from Proposition~\ref{prop:Hmb-Fmr51} and Lemmas~\ref{lem:gfr} and \ref{lem:bxd}.
\end{proof}

\begin{example}
As already observed in \cite{MarTeh} and Example~\ref{ex:SPA64}, the function $F\colon X^*\to\N$ defined by $F(\bfx)=|\bfx|$ is preassociative and standard. The function $g\colon\N\to X^*$ defined by $g(n)=a^n$ for some fixed $a\in X$ is a quasi-inverse of $F$. The function $H=g\circ F$, from $X^*$ to $X^*$, is then defined by $H(\bfx)=a^{|\bfx|}$ and the function $f=\restr{F}{\ran(H)}$, from $\ran(H)$ to $\N$, is defined by $f(a^n)=n$. In accordance with Theorem~\ref{factorizationPA}, we have $F=f\circ H$, where $f$ is injective and $H$ is associative and standard.
\end{example}

\begin{remark}
\begin{enumerate}
\item[(a)] The restriction of Theorem~\ref{factorizationPA} to standard functions having a $1$-determined range was obtained in \cite[Theorem~4.9]{MarTeh}. Here we have extended this factorization result to any preassociative function.

\item[(b)] It is noteworthy that, by making an appropriate choice of $g\in Q(F)$ in Lemma~\ref{lem:gfr}, Lemma~\ref{lem:bxd}, and Theorem~\ref{factorizationPA}, the function $\restr{H}{F^{-1}(\ran(F_k))}$ can always be made $k$-bounded for every $k\in\N$. Indeed, for every function $F\colon X^*\to X^*$, define the map $\ell\colon\ran(F)\to\N$ by
    $$
    \ell(y) ~=~ \min\{j\in\N : X^j\cap F^{-1}\{y\}\neq\varnothing\}.
    $$
    We say that a quasi-inverse $g$ of $F$ is \emph{length-optimized} if $g(y)\in X^{\ell(y)}$ for every $y\in\ran(F)$. Under AC we have $\varnothing\neq Q_{\ell}(F)\subseteq Q(F)$, where $Q_{\ell}(F)$ denotes the set of length-optimized quasi-inverses of $F$. Now, under the assumptions of Lemma~\ref{lem:gfr}, if $g\in Q_{\ell}(F)$, then for every $k\in\N$ the function $\restr{H}{F^{-1}(\ran(F_k))}$ is $k$-bounded. Indeed, if $\bfx\in F^{-1}(\ran(F_k))$, then $k\in\{j\in\N : X^j\cap F^{-1}\{F(\bfx)\}\neq\varnothing\}$ and therefore $|H(\bfx)|=|g(F(\bfx))|=\ell(F(\bfx))\leqslant k$.
\end{enumerate}
\end{remark}

Combining Proposition~\ref{prop:mB} and Theorem~\ref{factorizationPA}, we immediately derive the following corollary.

\begin{corollary}
Assume AC and let $m\in\N$. Any preassociative function $F\colon X^*\to Y$ having an $m$-determined range is completely determined by its parts of arity at most $m+1$, i.e., if $G\colon X^*\to Y$ is a preassociative function having an $m$-determined range and such that $G_i=F_i$, for $i=0,\ldots,m+1$, then $F=G$.
\end{corollary}

%\begin{proof}
%Let $F\colon X^*\to Y$ and $G\colon X^*\to Y$ be two function satisfying the stated conditions. By Theorem~\ref{factorizationPA}, there exist string-associative $m$-bounded maps $H\colon X^*\to X^*$ and $K\colon X^*\to X^*$ and injective functions $f=\restr{F}{\ran(H)}$ and $g=\restr{G}{\ran(K)}$ such that $F=f\circ H$ and $G=g\circ K$. It follows that $f=g$ and $H_k=f^{-1}\circ F_k=f^{-1}\circ G_k=K_k$ for $k=0,\ldots,m+1$, that is, $H=G$, and therefore $F=G$.
%\end{proof}

\begin{remark}
If $F \colon X^* \to Y$ is preassociative and has an $m$-determined range for some $m\in\N$, then by combining Eq.~\eqref{eqn:bla2} with Theorem~\ref{factorizationPA} we see that $F$ can be computed recursively from $F_0,\ldots,F_{m+1}$ by
$$
F_n(x_1\cdots x_n) ~=~ F((g\circ F_{n-1})(x_1\cdots x_{n-1}){\,}x_n),\qquad n\geqslant m+2,
$$
where $g\in Q(F)$ satisfies $\ran(g)\subseteq\bigcup_{k=0}^mX^k$.
% Note that the outer function $F$ in the right-hand side is always of arity $\leqslant m+1$.
For instance, let $F\colon\R^*\to\R\cup\{\varepsilon\}$ be the preassociative function having a $1$-determined range and such that $F(\varepsilon)=\varepsilon$, $F(x_1)=x_1$, and $F(x_1x_2)=x_1+x_2$ for all $x_1,x_2\in\R$. Then, the function $g\colon \R\cup\{\varepsilon\}\to \R^*$ defined by $g(\varepsilon)=\varepsilon$ and $g(x)=x$ for all $x\in\R$ is a quasi-inverse of $F$ and therefore we have
$$
F_3(x_1x_2x_3) ~=~ F((g\circ F_{2})(x_1x_2){\,}x_3) ~=~ x_1+x_2+x_3,
$$
and even $F_n(\bfx)=\sum_{i=1}^nx_i$ for every integer $n\geqslant 1$.
\end{remark}

We now provide necessary and sufficient conditions on the parts $F_0,\ldots,F_{m+1}$ for a function $F\colon X^*\to Y$ to be preassociative and have an $m$-determined range. The result is stated in Theorem~\ref{thm:we8ds6fs} below and follows from the next proposition.

\begin{proposition}\label{prop:34k3k4j5}
Assume AC and let $m\in\N$. A function $F\colon X^*\to Y$ is preassociative and has an $m$-determined range if and only if $\ran(F_{m+1})\subseteq\bigcup_{k=0}^m\ran(F_k)$ and there exists $g\in Q(F)$, with $\ran(g)\subseteq\bigcup_{k=0}^mX^k$, such that
\begin{enumerate}
\item[(a)] $F(x)=F(xH(\varepsilon))$ for all $x\in X$,

\item[(b)] $F(H(x\bfy)z)=F(xH(\bfy z))$ for all $x\in X$, $\bfy\in X^*$, and $z\in X$ such that $\length{x\bfy z}\leqslant m+2$,

\item[(c)] $F(\bfy z)=F(H(\bfy)z)$ for all $\bfy\in X^*$ and all $z\in X$,
\end{enumerate}
where $H=g\circ F$.
\end{proposition}

\begin{proof}
(Necessity) Let $F\colon X^*\to Y$ be preassociative and have an $m$-determined range. Then clearly $\ran(F_{m+1})\subseteq\ran(F)=\bigcup_{k=0}^m\ran(F_k)$. Let $g\in Q(F)$ such that $\ran(g)\subseteq\bigcup_{k=0}^mX^k$ and let $H=g\circ F$. By Lemma~\ref{lem:bxd}, $H$ is associative and $m$-bounded, and therefore conditions (a)--(c) hold by Proposition~\ref{prop:NSCfor-mBSA}.

(Sufficiency) Let $F\colon X^*\to Y$ be a function satisfying $\ran(F_{m+1})\subseteq\bigcup_{k=0}^m\ran(F_k)$ and conditions (a)--(c) for some $g\in Q(F)$ such that $\ran(g)\subseteq\bigcup_{k=0}^mX^k$. Since $H=g\circ F$ is $m$-bounded, by condition (c) we must have $\ran(F_{n})\subseteq\ran(F_{m+1})\subseteq\bigcup_{k=0}^m\ran(F_k)$ for every $n\geqslant 1$ and hence $F$ has an $m$-determined range.

Let us show that $F$ is preassociative. By Lemma~\ref{lem:bxd}, it suffices to show that $H=g\circ F$ is associative. By Proposition~\ref{prop:NSCfor-mBSA} it suffices to show that $H\circ H_k=H_k$ or equivalently $g\circ F\circ g\circ F_k=g\circ F_k$ for $k=0,\ldots,m+1$. This identity clearly holds by definition of $g$.
\end{proof}

\begin{theorem}\label{thm:we8ds6fs}
Assume AC and let $m\in\N$. For $k=0,\ldots,m+1$, let $F_k\colon X^k\to Y$ be functions. Then there exists a preassociative function $G\colon X^*\to Y$ having an $m$-determined range and such that $G_k=F_k$ for $k=0,\ldots,m+1$ if and only if $\ran(F_{m+1})\subseteq\bigcup_{k=0}^m\ran(F_k)$ and there exists $g\in Q(F)$, with $\ran(g)\subseteq\bigcup_{k=0}^mX^k$, such that conditions (a) and (b) of Proposition~\ref{prop:34k3k4j5} hold, where $H=g\circ F$. Such a function $G$ is then uniquely determined by $G(\bfy z)=G((g\circ G)(\bfy){\,}z)$ for every $\bfy z\in X^*$.
\end{theorem}

We end this section by giving equivalent conditions for a function $F\colon X^*\to Y$ to have an $m$-determined range. This result generalizes Proposition~2.4 in \cite{MarTeh}.

\begin{proposition}\label{prop:rw-we76r}
Assume AC, let $F\colon X^*\to Y$ be a function, and let $m\in\N$. The following assertions are equivalent.
\begin{enumerate}
\item[(i)] $F$ has an $m$-determined range.

\item[(ii)] There exists an $m$-bounded function $H\colon X^*\to X^*$ such that $F=F\circ H$.

\item[(iii)] There exists an $m$-bounded function $H\colon X^*\to X^*$, with $H_k=\restr{\id}{X^k}$ for $k=0,\ldots,m$, and a function $f\colon X^*\to Y$ such that $F=f\circ H$. In this case, $f_k=F_k$ for $k=0,\ldots,m$.

\item[(iv)] There exist functions $H\colon X^*\to X^*$ and $f\colon X^*\to Y$ such that $F=f\circ H$ and there exists a partition $\{A_0,\ldots,A_m\}$ of $X^*$ such that $\ran(\restr{H}{A_k})\subseteq X^k$ and $\restr{H}{A_k}=H_k\circ\restr{H}{A_k}$ for $k=0,\ldots,m$. In this case, $\restr{F}{A_k}=F_k\circ\restr{H}{A_k}$ for $k=0,\ldots,m$.

\item[(v)] There exists a function $H\colon X^*\to X^*$ having an $m$-determined range and a function $f\colon X^*\to Y$ such that $F=f\circ H$.
\end{enumerate}
\end{proposition}

\begin{proof}
(i) $\implies$ (ii) Follows from Lemma~\ref{lem:gfr}.

(ii) $\implies$ (iii) Modifying $H_k$ into $\restr{\id}{X^k}$ for $k=0,\ldots,m$ and taking $f=F$, we obtain $F=f\circ H$. We then have $F_k=f\circ H_k=f_k$ for $k=0,\ldots,m$.

(iii) $\implies$ (iv) The first part is trivial. We can take, e.g., $A_k=H^{-1}(X^k)$. Also, we have
$\restr{F_k\circ H}{A_k}=\restr{f\circ H_k\circ H}{A_k}=\restr{f\circ H}{A_k}=\restr{F}{A_k}$ for $k=0,\ldots,m$.

(iv) $\implies$ (v) If $y\in\ran(H)$, then there exists $k\in\{0,\ldots,m\}$ such that $y\in\ran(\restr{H}{A_k})\subseteq\ran(H_k)$. Hence $H$ has an $m$-determined range.

(v) $\implies$ (i)  Follows from the fact that the property of having an $m$-determined range is preserved under left composition with unary maps.
\end{proof}

\begin{corollary}
Let $m\in\N$ and let $F\colon X^*\to Y$ have an $m$-determined range. If $F_0,\ldots,F_m$ are injective, then there exists a unique $m$-bounded function $H\colon X^*\to X^*$ such that $F=F\circ H$.
\end{corollary}

\begin{proof}
By Proposition~\ref{prop:rw-we76r} there exists an $m$-bounded function $H\colon X^*\to X^*$ such that $F=F\circ H$. Also, there exists a partition $\{A_0,\ldots,A_m\}$ of $X^*$ such that $\restr{F}{A_k}=F_k\circ\restr{H}{A_k}$, or equivalently, $\restr{H}{A_k}=F_k^{-1}\circ\restr{F}{A_k}$ for $k=0,\ldots,m$. Hence $H$ is uniquely determined.
\end{proof}

%---------------------------------------------------------------------------------------------- Section 4
\section{Functions depending only on the length of the input}

We now consider the special class of string functions that depend only on the length of the input. Our aim is to characterize associativity and preassociativity within this class.

\begin{definition}
We say that a function $F \colon X^* \to X^*$ is \emph{weakly length-based} if for every $\bfx,\bfy\in X^*$ we have $\length{F(\bfx)}=\length{F(\bfy)}$ whenever $\length{\bfx}=\length{\bfy}$. We say that a function $F\colon X^*\to Y$ is \emph{length-based} if for every $\bfx,\bfy\in X^*$ we have $F(\bfx)=F(\bfy)$ whenever $\length{\bfx}=\length{\bfy}$.
\end{definition}

Note that $F$ is length-based if and only if there exists a map $\phi \colon \N \to X^*$ such that $F = \phi \circ \length{\cdot}$, i.e., $F(\bfx) = \phi(\length{\bfx})$ for all $\bfx \in X^*$.

\begin{example}
Any standard function $F\colon X^* \to X\cup\{\varepsilon\}$ such that $F(\varepsilon)=\varepsilon$ is weakly length-based.
It is easy to see that if $\phi$ satisfies $\length{\phi(n)} = n$ for all $n \in \N$, then the function $F \colon X^* \to X^*$ given by $F = \phi \circ \length{\cdot}$ is associative and standard.
For another example, let $\phi \colon \N \to X^*$ be any map satisfying $\length{\phi(0)} = 0$, $\length{\phi(1)} = 1$, $\length{\phi(2k)} = 4$, $\length{\phi(2k + 1)} = 5$, for all integers $k \geqslant 1$. Then $F \colon X^* \to X^*$ given by $F = \phi \circ \length{\cdot}$ is associative and standard. Finally, for every integer $k\geqslant 1$, the function $F\colon X^*\to\N$ defined by $F=\length{\cdot}\mod k$ is preassociative.
\end{example}

\begin{proposition}\label{prop:lengthassoc}
Let $F \colon X^* \to X^*$ be a function.
\begin{enumerate}
\item[(a)] If $F$ is associative and weakly length-based, then there is a map $\alpha\colon\N \to \N$ such that $\ran(F_k)\subseteq X^{\alpha(k)}$ for all $k\in \N$ and
\begin{equation}
\label{eq:g(n+k)}
\alpha(n + k) ~=~ \alpha(\alpha(n) + k), \qquad \text{for all $n, k \in \N$}.
\end{equation}
In this case, $F$ is standard if and only if $\alpha$ satisfies
\begin{equation}\label{eq:g(0)=0}
 \alpha(n)=0\quad \iff\quad n=0, \qquad \text{for all } n \in\N.
\end{equation}

\item[(b)] $F$ is associative and length-based if and only if  $F = \psi \circ \alpha \circ \length{\cdot}$ for some $\psi \colon \N \to X^*$ satisfying $\length{\psi(n)} = n$ for all $n \in \N$ and some $\alpha\colon\N \to \N$ satisfying (\ref{eq:g(n+k)}).
\end{enumerate}
\end{proposition}

\begin{proof}
(a) Since $F$ is weakly length-based, there is a function $\alpha\colon \N \to \N$ such that $\ran(F_k)\subseteq X^{\alpha(k)}$ for all $k\in \N$. By associativity, we have $F(\bfx \bfy) = F(F(\bfx) \bfy)$ for all $\bfx, \bfy \in X^*$ with $\length{\bfx} = n$ and $\length{\bfy} = k$. Since $\length{F(\bfx \bfy)}=\alpha(n+k)$ and $\length{ F(F(\bfx) \bfy)}=\alpha(\alpha(n)+k)$,  it follows that $\alpha(n + k) = \alpha(\alpha(n) + k)$ for all $n, k \in \N$.

The last part of the statement follows from the fact that $\varepsilon$ is the only zero-length string.

(b) (Necessity)
By (a) and since $F$ is length-based, there is a map $\alpha\colon\N \to \N$ satisfying (\ref{eq:g(n+k)}) and  some $\bfy_k\in X^{\alpha(k)}$ for every $k\in \N$ such that $F(\bfx)=\bfy_k$ for all $k \in \N$ and all $\bfx \in X^k$.

Together with associativity, this implies that if $\length{F(\bfx)} = \length{F(\bfy)}$, then
\[
F(\bfx)
~=~ F(F(\bfx))
~=~ \bfy_{\length{F(\bfx)}}
~=~ \bfy_{\length{F(\bfy)}}
~=~ F(F(\bfy))
~=~ F(\bfy).
\]
Therefore, we can decompose $F$ as $F = \psi \circ \alpha\circ \length{\cdot}$ for some $\psi \colon \N \to X^*$ such that $\length{\psi(n)} = n$ for all $n \in \N$.

(Sufficiency)
The function $F = \psi \circ \alpha \circ \length{\cdot}$ is clearly length-based.
In order to verify associativity, let $\bfx, \bfy, \bfz \in X^*$ with $\length{\bfx} = a$, $\length{\bfy} = b$, $\length{\bfz} = c$.
By condition \eqref{eq:g(n+k)} we have
\[
F(\bfx F(\bfy) \bfz)
~=~ f(\alpha(a + \alpha(b) + c))
~=~ f(\alpha(a + b + c))
~=~ F(\bfx \bfy \bfz).
\]
This completes the proof.
\end{proof}

\begin{remark}
It is not difficult to see that the converse statement of Proposition~\ref{prop:lengthassoc}(a) does not hold. As a counterexample, take $F\colon\{0,1\}^*\to\{0,1\}^*$ defined as $F(\varepsilon)=\varepsilon$ and $F(x_1\cdots x_n)=\overline{x}_1\cdots\overline{x}_n$ for every integer $n\geqslant 1$, where $\overline{x}=1-x$. We have $F\circ F\neq F$ and hence $F$ is not associative. However, we have $\ran(F_k)\subseteq X^{\alpha(k)}$ for all $k\in \N$, where $\alpha=\id$ satisfies Eq.~\eqref{eq:g(n+k)}.
\end{remark}

By Proposition~\ref{prop:lengthassoc}, the problem of characterizing the length-based associative  functions reduces to the problem of characterizing the functions $\alpha \colon \N \to \N$ satisfying \eqref{eq:g(n+k)}.
In what follows, we find an explicit description of such functions $\alpha$ (see Proposition~\ref{prop:g}).
We first need to establish a few auxiliary results.
We begin by reformulating condition \eqref{eq:g(n+k)} in order to simplify the analysis.

\begin{lemma}
\label{lem:conditions}
Condition \eqref{eq:g(n+k)} is equivalent to
%  \eqref{eq:gg=g} and
\begin{align}
\label{eq:gg=g}
& \alpha(\alpha(n)) = \alpha(n)
&&
\text{and} \\
\label{eq:nn'k}
& \alpha(n) = \alpha(n') \implies \alpha(n + k) = \alpha(n' + k),
&&\text{for all $n, n', k \in \N$.}
\end{align}
\end{lemma}

\begin{proof}
Condition \eqref{eq:gg=g} is a special case of \eqref{eq:g(n+k)} with $k = 0$.
Under the assumption that $\alpha(n) = \alpha(n')$, it follows from condition~\eqref{eq:g(n+k)} that
\[
\alpha(n + k) ~=~ \alpha(\alpha(n) + k) ~=~ \alpha(\alpha(n') + k) ~=~ \alpha(n' + k).
\]
Condition \eqref{eq:g(n+k)} follows from \eqref{eq:gg=g} and \eqref{eq:nn'k} by taking $n' = \alpha(n)$.
\end{proof}

A function $\alpha \colon \N \to \N$ is  \emph{$(n_1,p)$-periodic} if for all $n \geqslant n_1$ it holds that $\alpha(n) = \alpha(n + p)$.
It is clear that if $\alpha$ is $(n_1,p)$-periodic, then it is $(n'_1,p')$-periodic for every $n'_1 \geqslant n_1$ and for every multiple $p'$ of $p$.

The following lemma is folklore. We provide a proof for the sake of self-contained{\-}ness.

\begin{lemma}
\label{lem:ultper}
If the function $\alpha \colon \N \to \N$ is $(n_1,p_1)$-periodic and $(n_2,p_2)$-periodic, then $\alpha$ is $(\min(n_1, n_2),\gcd(p_1, p_2))$-periodic.
\end{lemma}

\begin{proof}
Assume, without loss of generality, that $n_1 \leqslant n_2$.
Let $d = \gcd(p_1, p_2)$.
We need to show that $\alpha(n) =  \alpha(n + d)$ whenever $n \geqslant n_1$.

By B\'ezout's lemma, there exist integers $c_1$ and $c_2$ such that $d = c_1 p_1 + c_2 p_2$. Note that one of $c_1$ and $c_2$ is nonnegative and the other is nonpositive.
Consider first the case that $c_1 \leqslant 0$ and $c_2 \geqslant 0$.
Let $k \in \N$ be a large enough integer such that $(k p_2 + c_1) p_1 \geqslant n_2 - n_1$.
Then, for $n \geqslant n_1$, we have
\begin{align*}
\alpha(n + d)
&~=~ \alpha(n + c_1 p_1 + c_2 p_2 + k p_1 p_2) \\
&~=~ \alpha(n + (k p_2 + c_1) p_1 + c_2 p_2)
~=~ \alpha(n + (k p_2 + c_1) p_1)
~=~ \alpha(n),
\end{align*}
where the first equality holds by B\'ezout's identity and because $\alpha$ is $(n_1,p_1)$-periodic;
the second equality is the result of a simple algebraic rearrangement;
the third equality holds because $n + (k p_2 + c_1) p_1 \geqslant n_2$ and $\alpha$ is $(n_2,p_2)$-periodic;
and the last equality holds because $\alpha$ is $(n_1,p_1)$-periodic.

In the case when $c_1 \geqslant 0$ and $c_2 \leqslant 0$ we choose $k$ in such a way that $(k p_1 + c_2) p_2 \geqslant n_2 - n_1$. Then a similar argument shows that if $n \geqslant n_1$ then $\alpha(n + d) = \alpha(n)$ holds also in this case.
\end{proof}

\begin{lemma}
\label{lem:ultperdiff}
Assume that $\alpha \colon \N \to \N$ satisfies conditions \eqref{eq:gg=g} and \eqref{eq:nn'k}.
If $\alpha(n) \neq n$, then $\alpha$ is $(n_0, n'_0 - n_0)$-periodic, where $n_0 = \min \{n, \alpha(n)\}$ and $n'_0 = \max \{n, \alpha(n)\}$.
\end{lemma}

\begin{proof}
By \eqref{eq:gg=g}, we have $\alpha(\alpha(n)) = \alpha(n)$; hence $\alpha(n_0) = \alpha(n'_0)$.
It follows from \eqref{eq:nn'k} that for all $k \in \N$,
\[
\alpha(n_0 + k + (n'_0 - n_0)) ~=~ \alpha(n'_0 + k) ~=~ \alpha(n_0 + k).
\qedhere
\]
\end{proof}

We are now in position to describe how the length of the output depends on the length of the input in a length-based associative function.
\begin{proposition}
\label{prop:g}
Let $\alpha \colon \N \to \N$. The following conditions are equivalent.
\begin{enumerate}
\item[(i)] $\alpha$ satisfies conditions \eqref{eq:gg=g} and \eqref{eq:nn'k}.
\item[(ii)] Either $\alpha$ is the identity function on $\N$ or there exist integers $n_1\geqslant 0$ and $\ell> 0$ such that
   \begin{enumerate}
   \item[(a)] $\alpha(n) = n$ whenever $0 \leqslant n < n_1$,
   \item[(b)] $\alpha$ is $(n_1,\ell)$-periodic,
  \item[(c)] $\alpha(n) \geqslant n$ and $\alpha(n) \equiv n \pmod{\ell}$ whenever $n_1 \leqslant n < n_1 + \ell$.
   \end{enumerate}
\end{enumerate}
In addition, $\alpha$ satisfies condition \eqref{eq:g(0)=0} if and only if $\alpha$ is the identity function on $\N$ or $\alpha$ satisfies conditions (a)--(c) with $n_1>0$.

\end{proposition}

\begin{proof}
(i) $\implies$ (ii).
If $\alpha$ is not the identity function, then
the set $D = \{n \in \N : \alpha(n) \neq n\}$ is nonempty.
Let $g(D) = \{\alpha(n) : n \in D\}$, let $n_1$ be the minimum element of $D \cup \alpha(D)$, and let $\ell$ be the minimum of the set $L = \{ \lvert n - \alpha(n) \rvert : n \in D \}$.
In view of Lemmas \ref{lem:ultper} and \ref{lem:ultperdiff}, $\alpha$ is $(n_1,\ell)$-periodic.
Moreover, $\alpha(n) = n$ whenever $n < n_1$ and $n_1>0$ if $\alpha(0)=0$.

Let $n \in \{n_1, \dots, n_1 + \ell - 1\}$.
Suppose, on the contrary, that $\alpha(n) < n$.
If $\alpha(n) < n_1$ then $n_1$ would not be the minimum element of $D \cup \alpha(D)$, a contradiction.
If $n_1 \leqslant \alpha(n) < n$ then $\lvert n - \alpha(n) \rvert < \ell$, which contradicts the minimality of $\ell$ in the set $L$.
We conclude that $\alpha(n) \geqslant n$.

Suppose then, on the contrary, that $\alpha(n) \not\equiv n \pmod{\ell}$.
Then $\alpha(n) - n = q \ell + r$ for some $q \geqslant 0$ and $0 < r < \ell$.
Since $\alpha$ is $(n_1,\ell)$-periodic, we have
\[
\alpha(n + q \ell) ~=~ \alpha(n) ~=~ n + q \ell + r.
\]
By Lemma~\ref{lem:ultperdiff}, this contradicts again the minimality of $\ell$ in the set $L$.
We conclude that $\alpha(n) \equiv n \pmod{\ell}$.

Finally, note that if $\alpha$ satisfies condition \eqref{eq:g(0)=0}, then necessarily $n_1 > 0$.

(ii) $\implies$ (i).
The identity function on $\N$ clearly satisfies conditions \eqref{eq:g(0)=0}, \eqref{eq:gg=g}, and \eqref{eq:nn'k}. If $\alpha$ satisfies conditions (a)--(c) with $n_1 > 0$, then
$\alpha$ also satisfies condition \eqref{eq:g(0)=0}.
Assume that $\alpha$ is not the identity function.
%Then, by (a), $g$ satisfies condition \eqref{eq:g(0)=0}.

If $0 \leqslant n \leqslant n_1 - 1$, then $\alpha(n) = n$ by (a); hence $\alpha(\alpha(n)) = \alpha(n)$.
If $n \geqslant n_1$, then $n \equiv m \pmod{\ell}$ for some $m \in \{n_1, \dots, n_1 + \ell - 1\}$.
By (b), $\alpha(n) = g(m)$, and by (c), $\alpha(m) \geqslant n_1$ and $\alpha(m) \equiv m \pmod{\ell}$.
Consequently, $\alpha(\alpha(n)) = \alpha(\alpha(m)) = \alpha(m) = \alpha(n)$.
Thus, $\alpha$ satisfies condition \eqref{eq:gg=g}.

Assume then that $\alpha(n) = \alpha(n')$. If $\alpha(n) \leqslant n_1 - 1$, then $n = n'$ and $\alpha(n + k) = \alpha(n' + k)$ holds trivially for all $k \in \N$.
If $\alpha(n) \geqslant n_1$, then both $n$ and $n'$ are greater than or equal to $n_1$ and $n \equiv n' \pmod{\ell}$.
Consequently, for all $k \in \N$, it holds that $n + k \equiv n' + k \pmod{\ell}$ and $\alpha(n + k) = \alpha(n' + k)$ by (b).
\end{proof}

We now apply Theorem \ref{factorizationPA} to characterize length-based preassociative functions.

\begin{proposition}
Assume AC and let $F\colon X^* \to Y$ be a function. The following conditions are equivalent.
\begin{enumerate}
\item[(i)] F is preassociative and length-based.
\item[(ii)] There exist  functions $\mu\colon \N \to X^*$ and $f\colon X^* \to Y$ such that $F=f\circ\mu\circ \length{\cdot}$, where $\restr{f}{\ran(\mu\circ\length{\cdot})}$ is injective and the function $\alpha\colon\N\to \N$ defined by $\alpha(n)=\length{\mu(n)}$ satisfies condition (ii) of Proposition \ref{prop:g}.
\end{enumerate}
Moreover, the equivalence still holds if we add the condition that $F$ is standard in (i) and the condition $\mu(0)=\varepsilon$ in (ii).
\end{proposition}

\begin{proof}
(i) $\implies$ (ii). If $F$ is preassociative, then by Theorem \ref{factorizationPA} there is a associative function $H\colon X^* \to X^*$ and a map $f\colon X^*\to Y$ such that $F=f\circ H$ and $\restr{f}{\ran(H)}$ is injective. If $F$ is length-based, then so is $H$ and, by Propositions~\ref{prop:lengthassoc} and \ref{prop:g}, we have $H=\psi\circ \alpha \circ \length{\cdot}$ for some function $\alpha\colon \N \to \N$ satisfying condition (ii) of Proposition  \ref{prop:g}  and some function $\psi\colon \N \to X^*$ such that $\length{\psi(n)}=n$ for all $n\in \N$. It suffices to set $\mu=\psi\circ \alpha$ to obtain the desired result.

(ii) $\implies$ (i). The function $F=f\circ\mu\circ\length{\cdot}$ is clearly length-based. Moreover, according to Propositions~\ref{prop:lengthassoc} and \ref{prop:g}, the function $\mu\circ \length{\cdot}$ is associative. By Theorem \ref{factorizationPA}, the function $F$ is preassociative.

The last part of the statement is again a consequence of Propositions~\ref{prop:lengthassoc} and \ref{prop:g}.
\end{proof}

%---------------------------------------------------------------------------------------------- Acknowledgments
\section*{Acknowledgments}

This work was developed within the FCT Project PEst-OE/MAT/UI0143/2014 of CAUL, FCUL. It is also supported by the internal research project F1R-MTH-PUL-12RDO2 of the University of Luxembourg. The authors would like to thank the anonymous reviewer for timely and helpful suggestions.

%---------------------------------------------------------------------------------------------- Bibliography

\end{document}